\documentclass{article}
\usepackage{amsmath, amssymb, mathtools, amsthm, tikz, graphicx, multicol, biblatex, setspace}
\usetikzlibrary{fit}
\usetikzlibrary{decorations.pathmorphing}
\usepackage[margin=1in]{geometry}
\usepackage{xcolor, soul, caption, capt-of}

\usepackage[shortlabels]{enumitem}
\newtheorem{theorem}{Theorem}[section]
\newtheorem{lemma}[theorem]{Lemma}

\newtheorem{corollary}[theorem]{Corollary}

\numberwithin{figure}{section}

\title{A characterization of all graphs cospectral to the double star $P_2(1,n)$}
\author{Emily Barranca \& Michael D. Barrus }
\date{August 2024}

\begin{document}


\begin{center}
{\Large A characterization of all graphs cospectral to the double star $P_2(1,n)$}
\vskip.20in
\begin{multicols}{2}
    Emily Barranca\\
    {\footnotesize Department of Mathematics and\\ Computer Science\\ St. Mary's College of Maryland\\ St. Mary's City, MD }
    \columnbreak 
    
    Michael D. Barrus\footnote{Posthumously} \\
    {\footnotesize Department of Mathematics and \\Applied Mathematical Sciences\\
University of Rhode Island\\ Kingston, RI  }
\end{multicols}

\end{center}

\vspace{0.2in}

\textbf{Abstract}
 We examine the adjacency spectrum of trees with diameter three, also referred to as double stars. 
 Using $P_2(a,b)$ to denote a double star with $ a$ and $b$ leaves at its respective endpoints, we discuss graphs which are cospectral to double stars for various parameters $a$ and $b$. In particular, we give constructions for graphs cospectral to $P_2(1,2k)$ for integers $k$.  Lastly, we show that the double star $P_2(1,n)$  is determined by its spectrum when $n$ is odd. That is, if a graph $G$ cospectral to $P_2(1,n)$ for odd $n$, then $G$ is isomorphic to $P_2(1,n)$.

\section{Introduction}\label{sec: intro}

Let $G = (V(G), E(G))$ be a graph. In this paper, all graphs discussed are simple and undirected. Recall that a subgraph $G' = (V(G'), E(G'))$ is \textit{induced} by the vertex set $V(G') \subseteq V(G)$ if for all pairs $u,v \in V(G')$ such that $\{u,v\} \in E(G)$, we have that $\{u,v\} \in E(G')$. We denote this subgraph $G[V(G')]$. To denote the subgraph $H'$ of graph $H$ induced by the set $V(H) \setminus \{v_1, v_2, \hdots v_n\}$, we write $H' = H \setminus  \{v_1, v_2, \hdots v_n\}$. \\

 Let $A(G)$ be the adjacency matrix of a graph $G$. Here the \textit{spectrum} of a graph refers to the adjacency spectrum, that is, the multiset of eigenvalues of the graph's adjacency matrix. Equivalently, the spectrum of $G$ is the set of roots of the \textit{characteristic polynomial} $\phi(G) = \det(A(G)-xI)$ together with their multiplicities. The adjacency spectrum determines some structural properties of a graph, including whether a graph is bipartite, the number of vertices and edges in a graph, and information on the diameter of a graph \cite{bh}. Suppose $G$ is a graph with two connected components $H$ and $K$. We denote the disjoint union of connected components as $G = H+K$, and note that $\phi(G)= \phi(H+K) = \phi(H) \phi(K)$ \cite{text2}.
 \\
 
 Two graphs are \textit{cospectral} when they have the same spectrum. Cospectral graphs must share structural properties which are given by the spectrum, but not every structural property of a graph is determined by the graph's spectrum. For example, consider Figure~\ref{fig: saltire}, depicting the graphs $K_{1,4}$ and $C_4 + K_1$. Both graphs have the spectrum $\{2^{(1)}, 0^{(3)}, -2^{(1)}\}$, where superscripts denote algebraic multiplicities of each eigenvalue. Since $K_{1,4}$ is connected and  $C_4 + K_1$ is not, the spectrum of a graph cannot determine the number of components in the graph.  Similarly, $K_{1,4}$ is a tree while the nontrivial component $C_4$ of the other graph is not, so the spectrum cannot determine whether a graph is a tree or forest. On the other hand, a graph is bipartite if and only if its spectrum is symmetric about zero. That is, a graph is bipartite if and only if $\lambda$ is an eigenvalue of $G$ implies $-\lambda$ is an eigenvalue of $G$ with the same multiplicity, so whether a graph is bipartite can be concluded from the spectrum \cite{alggt}.\\

A graph $G$ is said to be \textit{determined by its spectrum} (or \textit{DS}) if any graph $G'$ cospectral to $G$ is isomorphic to $G$. Some examples of graphs which are DS are empty graphs $E_n$, path graphs $P_n$, complete graphs $K_n$, and the disjoint union of complete components. Other examples include balanced complete bipartite graphs $K_{n,n}$ and the line graphs $L(K_n)$ and $L(K_{n,n})$ \cite{vandam}. \\

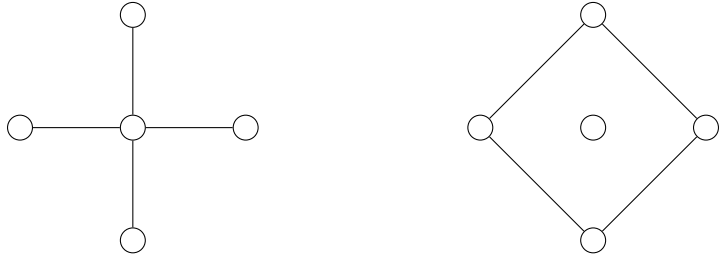
\begin{figure}[t]
    \centering
        \begin{tikzpicture}[main/.style = {draw, circle}] 
        \node[main] (0) at (0,0){};
        \node[main](1) at (0, 1.5){};
        \node[main](2) at (1.5,0) {};
        \node[main](3) at (0,-1.5){};
        \node[main](4) at (-1.5,0){};
        \draw (0)--(1);
        \draw (0)--(2);
        \draw (0)--(3);
        \draw (0)--(4);
    \end{tikzpicture}
    \hspace{1in} 
    \begin{tikzpicture}[main/.style = {draw, circle}] 
        \node[main] (0) at (0,0){};
        \node[main](1) at (0, 1.5){};
        \node[main](2) at (1.5,0) {};
        \node[main](3) at (0,-1.5){};
        \node[main](4) at (-1.5,0){};
        \draw (2)--(1);
        \draw (3)--(2);
        \draw (4)--(3);
        \draw (1)--(4);
    \end{tikzpicture}
    \caption{The graph $K_{1,4}$ pictured on the left is cospectral to $C_4+K_1$, pictured on the right.  }
    \label{fig: saltire}
\end{figure}
\begin{figure}[t]
    \centering
    \begin{tikzpicture}[main/.style={circle, draw, scale = 0.6}, scale = 1.2]
        \node[main] (a1) at (-1,1.25){$a_1$};
        \node[main] (a2) at (-1, 0.5){$a_2$};
        \node[circle] (adots) at (-1,-0.25){$\vdots$};
        \node[main] (am) at (-1,-1){$a_m$};
        \node[circle, draw] (p1) at (0,0){};
        \node[circle, draw] (p2) at (1,0){};
        \node[main] (b1) at (2,1.25){$b_1$};
        \node[main] (b2) at (2, 0.5){$b_2$};
        \node[circle] (bdots) at (2,-0.25){$\vdots$}; 
        \node[main] (bn) at (2,-1){$b_n$};
        \draw (a1) -- (p1);
        \draw (a2) -- (p1);
        \draw (am) -- (p1);
        \draw (p1) -- (p2);
        \draw (p2) -- (b1);
        \draw (p2) -- (b2);
        \draw (p2) -- (bn);
    \end{tikzpicture}
    \caption{The double star $P_2(m,n)$}
    \label{fig: double star p2ab}
\end{figure}
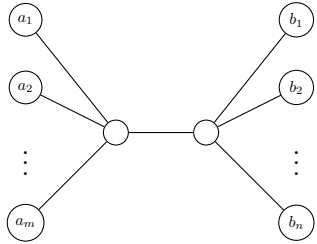

We use $P_2(a,b)$ to denote the \textit{double star} obtained by adding an edge between the center vertices of $K_{1,a}$ and $K_{1,b}$, as shown in in Figure~\ref{fig: double star p2ab}. We are interested in exploring the question: which double stars are DS? Schwenk showed in \cite{schwenk_trees}  that almost all trees are not DS. This result, however, relies on a subgraph with diameter 7 contained in almost all trees. Since double stars have diameter at most three, the question remains open for double stars.  \\

We start by considering the smallest cases of double stars. Note that $P_2(0,n)$ is isomorphic to the star graph $K_{1,n+1}$, which has spectrum $\{ \sqrt{n+1}^{(1)}, 0^{(n)}, -\sqrt{n+1}^{(1)}\}$. Then $P_2(0,n)$ is DS if and only if $n+1$ is prime. Let $G$ be a graph cospectral to $K_{1,n+1}$; then $G$ has only one positive eigenvalue and hence one nontrivial component. Furthermore $G$ is $P_4$-free and bipartite, and thus $G$ is a complete bipartite graph. If $n+1$ is composite, that is $n+1 = xy$, then the star $K_{1,n+1} = K_{1, xy}$ is cospectral with the graph $K_{x,y} + (x-1)(y-1)K_1$. Conversely, if $n+1$ is prime, there cannot exist a nonisomorphic complete bipartite graph with the spectrum $\{ \sqrt{n+1}^{(1)}, 0^{(n)}, -\sqrt{n+1}^{(1)}\}$. 
\\

Next, we shift our focus to slightly larger double stars. The double star $P_2(1,n)$ can also be classified as a starlike tree. A tree is said to be \textit{starlike} if exactly one vertex has degree greater than two. It was shown in \cite{starlike} that no pair of nonisomorphic starlike trees can be cospectral. This does not, however, explore cases where a starlike tree is cospectral to a graph that is not a starlike tree. 
\\

 A brute force search for cospectral graphs on graphs up to 13 vertices shows that $P_2(1,1)$, $P_2(1,2)$, $P_2(1,3)$, $P_2(1,5)$, $P_2(1,7)$, and $P_2(1,9)$ are DS double stars. Interestingly, the only other double star that is DS on 13 or fewer vertices is $P_2(5,5)$. We conjecture that very few double stars are DS. In contrast, it is widely believed that almost all graphs are DS. In Section~\ref{sec: pairs}, we review results on double stars which are not DS. In particular, we prove that $P_2(1,n)$ is not DS when $n$ is even. Section~\ref{sec: 1,a} then shows that if $n$ is odd, then $P_2(1,n)$ is DS.  \\

\section{Non-DS Double Stars} \label{sec: pairs}

In this section, we discuss several families of graphs that are cospectral to a double star $P_2(a,b)$ given various conditions on $a$ and $b$. For reference, we state the characteristic polynomial and spectrum of a double star here. \begin{theorem}\textup{\cite{double_star}}\label{double_star_spec}
\begin{enumerate}
    \item[\textup{(i)}]  The characteristic polynomial for the double star is 
       \begin{align*}
           \phi(P_2(a,b)) &= \phi(K_{1,a}) \phi(K_{1,b})-x^{a+b}\\
           &= x^{a+b+2} -(a+b+1)x^{a+b} +abx^{a+b-2}.
       \end{align*}
    \item[\textup{(ii)}] The spectrum of the double star $P_2(a,b)$ is $\{\lambda_1^{(1)}, \lambda_2^{(1)}, 0^{(a+b-2)}, - \lambda_2^{(1)},  -\lambda_1^{(1)}\}$
    where \begin{align*}
        \lambda_1 &= \sqrt{\frac{(a+b+1) + \sqrt{(a-b)^2+2(a+b)+1}}{2}};\\
        &\\
        \lambda_2 &= \sqrt{\frac{(a+b+1) - \sqrt{(a-b)^2+2(a+b)+1}}{2}}.\\
    \end{align*}
\end{enumerate}
\end{theorem}

Additionally, to prove Theorem~\ref{thm: P_2(1,even)}, we will use the recursive construction for the characteristic polynomial of a graph given by Schwenk: 

\begin{theorem} \label{thm: schwenk_charpoly} \textup{\cite{schwenk_charpoly}}
    Let $v$ be a vertex in graph $G$ and let $\mathcal{C}(v)$ be the collection of cycles containing $v$. Then 
 $$\phi (G) = x\phi(G \setminus \{v\}) - \sum_{\{u,v\}\in E(G)} \phi(G \setminus \{v,u\}) - 2\sum_{Z \in \mathcal{C}(v)} \phi (G\setminus V(Z)). $$
 \end{theorem}

A full proof of the following theorem can be found in \cite{diss}. This proof builds graphs cospectral to $P_2(a,b)$ to show it is not DS. Many of these constructions contain complete bipartite graphs attached to other complete bipartite graphs at a shared vertex. We suspect that for any graph cospectral to a double star, this structural property will appear.

\begin{theorem}\label{thm:pairs}
For the following conditions on $a$ and $b$, $P_2(a,b)$ is not DS: 
    \begin{enumerate}
    \begin{multicols}{3}
    \item[\textup{(i)}] $a=2$ and $b$ is odd
    \item[\textup{(ii)}] $a=3$ and $b$ is even
    \item[\textup{(iii)}] $a=4$ and $b \ge 1$
    \item[\textup{(iv)}] $a=8$ and $b \ge 2$
    \item[\textup{(v)}] $a=9$ and $b \ge 2$
    \item[\textup{(vi)}] $a = 2n-1$ and $b = n$ 
    \item[\textup{(vii)}] $a=2r$ and $b=r-1$
    \item[\textup{(viii)}] $a = 2m$, $b = m+1$, and $m\ge 2$. 
    \end{multicols}
\end{enumerate}
\end{theorem}

In the remainder of this section, we will prove the following theorem. 
\begin{theorem}\label{thm: P_2(1,even)}
Let $k >1$. The double star $P_2(1,2k)$ is not DS. Furthermore, if $k$ is even, then there are at least two nonisomorphic graphs cospectral to $P_2(1,2k)$.
\end{theorem}

\begin{proof}
Consider the double star $P_2(1,2a)$ for $a \ge 2$, which has characteristic polynomial \\$\phi(P_2(1,2a)) =x^{2a+3}-(2a+2)x^{2a+1} +2ax^{2a-1} $ by Theorem~\ref{double_star_spec}. \\

Let $A_a$ be the graph obtained from $K_{2,a}$ by adding two pendent vertices adjacent to a vertex in the partite set of order two, with labels given in Figure~\ref{fig: Aa Ba}. By Theorem~\ref{thm: schwenk_charpoly}, consider successively removing vertices $u_1$ and $u_2$. Then 
    \begin{align*}
        \phi(A_a)&=x\phi(A_a\setminus \{u_1\}) - \phi(A_a\setminus \{u_1, v_2\})\\
        &=x(x\phi(A_a\setminus \{u_1, u_2\})- \phi(A_a\setminus \{u_1, u_2,v_2\}) - \phi(A_a\setminus \{u_1,v_2\})\\
        &= x(x\phi(K_{2,a})- \phi(K_{1,a})) - \phi(K_1+K_{1,a})\\
        &= x(x^{a+1}(x^2-2a) -(x^{a+1} - ax^{a-1}))-x(x^{a+1} - ax^{a-1})\\
        &= x^{a+4}-(2a+2)x^{a+2}+2ax^a.
    \end{align*}
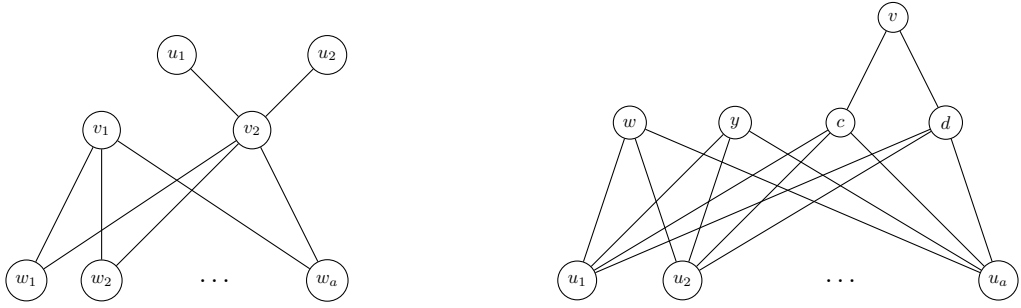
\begin{figure}[t]
    \centering
    \begin{tikzpicture}[main/.style={circle, draw, scale = 0.7}]
             \node[main] (a1) at (-1,2){$v_1$};
             \node[main] (a2) at (1,2){$v_2$};
             \node[main] (b1) at (0,3){$u_1$};
             \node[main] (b2) at (2,3){$u_2$};
             \node[main] (c1) at (-2,0){$w_1$};
             \node[main] (c2) at (-1, 0){$w_2$};
             \node[circle] (dots) at (0.5, 0){$\hdots$}; 
             \node[main] (ca) at (2,0){$w_a$};
             \draw (a1) -- (c1);
             \draw (a1) -- (c2); 
             \draw (a1) -- (ca);
             \draw (a2) -- (c1);
             \draw (a2) -- (c2); 
             \draw (a2) -- (ca);
             \draw (a2) -- (b1);
             \draw (a2) -- (b2); 
         \end{tikzpicture}
         \hspace{1in}
         \begin{tikzpicture}[main/.style = {circle, draw, scale = 0.7}, scale = 0.7]
        \node[main] (v) at (5,3){$v$};
        \node[main] (a) at (0,1){$w$};
        \node[main] (b) at (2,1){$y$}; 
        \node[main] (c) at (4,1){$c$};
        \node[main] (d) at (6,1){$d$};
        \node[main] (1) at (-1, -2){$u_1$};
        \node[main] (2) at (1,-2){$u_2$};
        \node[circle] (dots) at (4,-2){$\hdots$};
        \node[main] (m) at (7,-2){$u_a$};
        \draw (v) -- (c);
        \draw (v) -- (d);
        \draw (a) -- (1);
        \draw (a) -- (2);
        \draw (a) -- (m); 
        \draw (b) -- (1);
        \draw (b) -- (2);
        \draw (b) -- (m); 
        \draw (c) -- (1);
        \draw (c) -- (2);
        \draw (c) -- (m); 
        \draw (d) -- (1);
        \draw (d) -- (2);
        \draw (d) -- (m); 
    \end{tikzpicture}
    \caption{The graphs $A_a$ and $B_a$.}
    \label{fig: Aa Ba}
\end{figure}
        Thus for $a\ge2$, the double star $P_2(1, 2a)$ is cospectral to the graph $A_a + (a-1)K_1$.\\

We give another construction for a nonisomorphic graph cospectral to the double star $P_2(1,4a)$ for integers $a \ge 1$. 
By Theorem~\ref{double_star_spec}, the double star $P_2(1,4a)$ has characteristic polynomial \\$\phi(P_2(1,4a)) = x^{4a+3} - (4a+2)x^{4a+1} +4ax^{4a-1}$.\\

Let $B_a$ be the graph obtained by taking a copy of the complete bipartite graph $K_{4,a}$ and adding a vertex adjacent to two of the vertices in the partite set with order four, with labels given in Figure~\ref{fig: Aa Ba}.\\

 We will use Theorem~\ref{thm: schwenk_charpoly} to find $\phi(B_a)$, removing the vertex $v$. Thus we have:
    \begin{align*}
        \phi(G) &= x\phi(G \setminus \{v\}) - 2\phi(G\setminus \{v,c\}) - 2a\phi(G\setminus\{v,c,u_1,d\}) -2 \cdot 2a(a-1)\phi(G\setminus \{v,d,u_1,w,u_2, c\})\\
        &- 2 \cdot 2a(a-1)(a-2)\phi(G\{v,d,u_1,w,u_2,y,u_3,c\})\\
        &= x\phi(K_{4,a})-2\phi(K_{3,a}) - 2a\phi(K_{2,a-1}) - 4a(a-1)\phi(K_{1,a-2}) -4a(a-1)(a-2)\phi((a-3)K_1)\\
        &= x(x^{a+4} - 4ax^{a+2}) - 2(x^{a+3} - 3ax^{a+1})-2a(x^{a+1} -2(a-1)x^{a-1}) - 4a(a-1)(x^{a-1} - (a-2)x^{a-3})\\
        &- 4a(a-1)(a-2)x^{a-3}\\
        &= x^{a+5} -(4a+2)x^{a+3} +(4a)x^{a+1}.
    \end{align*}

We conclude that the graph $P_2(1,4a)$ is cospectral to the graph $B_a + (3a-2)K_1$.\\
\end{proof}
 We will show in the next section that these are the only graphs cospectral to $P_2(1,n)$. That is, if $n=4$ or if $n \equiv 2 \mod 4$ and $n \ne 2$, then there exists exactly one nonisomorphic graph cospectral to $P_2(1,n)$. If $n \equiv 0 \mod 4$ and $n$ is not equal to 0 or 4, then there exist exactly two nonisomorphic graphs cospectral to $P_2(1,n)$. The double star $P_2(1,n)$ is DS otherwise. 


\section{The double stars $P_2(1,n)$ for odd $n$} \label{sec: 1,a}

We continue our discussion of double stars by showing that for odd $n$, the double star $P_2(1,n)$ is DS. \\

As stated in Section~\ref{sec: intro}, we conjecture that very few other double stars are DS. The results in Section~\ref{sec: pairs} show that many types of double stars are not DS; a computational search for graphs cospectral to double stars yields many more examples.
In contrast, we noted earlier that each of $P_2(1,1)$, $P_2(1,3)$, $P_2(1,5)$, $P_2(1,7)$, and $P_2(1,9)$ is DS.  \\

We will reference the following theorems: 
 \begin{theorem} \label{sachs} \textup{\cite{text2}} 
     For a graph $M$, let $c(M)$ denote the number of cycles in $M$, and let $k(M)$ be the number of components of $M$. Let $G$ be a graph, and let $\mathcal{G}_n$ be the set of $n$-vertex subgraphs of $G$ in which each component is a cycle or is isomorphic to $K_2$. 
     Then the characteristic polynomial $\phi(G)$ of $G$ is given by 
     $$\phi(G) = \sum_{n=0}^{|V(G)|} a_n x^{|V(G)|-n},$$
     where 
     $$a_n = \sum_{H \in \mathcal{G}_n} (-1)^{k(H)}2^{c(H)}.$$
\end{theorem}

\begin{theorem}\label{thm: interlacing} \textup{\cite{bh}}
    Let $A$ be a graph on $n$ vertices and $B$ an induced subgraph of $A$ on $m$ vertices. Suppose the eigenvalues of $A$ are given by $\lambda_1 \ge \hdots \ge \lambda_n$ and the eigenvalues of $B$ are given by $\theta_1 \ge \hdots \ge \theta_m$. Then the eigenvalues of $B$ \textit{interlace} the eigenvalues of $A$. That is 
    $$\lambda_i \ge \theta_i \ge \lambda_{n-m+i} \text{ for } i\in \{1,\hdots,m\}.$$
\end{theorem}

In this section, suppose that $G$ is a graph having the same spectrum as $P_2(1,n)$, where by Theorem~\ref{double_star_spec}, $\phi(G) = \phi(P_2(1,n)) = x^{n+3} - (n+2)x^{n+1} +nx^{n-1}$. We show that the eigenvalues force special structure on $G$.

\begin{lemma}\label{thm: lambda2 less 1}
    The graph $G$ has second largest eigenvalue strictly less than 1.
\end{lemma}
\begin{proof}
    Recall from Theorem~\ref{double_star_spec} that the second largest eigenvalue $\lambda_2$ of $P_2(1,n$) is given by 
    $$\lambda_2 = \sqrt{\frac{n+1 - \sqrt{n^2+4}}{2}}. $$
   
    Standard techniques from calculus show that this function is increasing in $n$ and bounded above by 1. 
\end{proof}

\begin{corollary} \label{cor: forbidden subgraphs}
    A graph $G$ cospectral to $P_2(1,n)$ has no induced subgraph isomorphic to $2K_2$, $R$,  or $P_2(2,2)$, where $R$ is the graph given in Figure~\ref{fig: C4(0,0,1,1)}. 
\end{corollary}
\begin{figure}[t]
    \centering
    \begin{tikzpicture}[main/.style= {circle, draw, scale =1.2}, scale =1.5]
        \node[main] (a) at (0,0){};
        \node[main] (b) at (1,0){};
        \node[main] (c) at (1,1){};
        \node[main] (d) at (0,1){};
        \node[main] (e) at (0,2){};
        \node[main] (f) at (1,2){};
        \draw (a) -- (b);
        \draw (b) -- (c);
        \draw (c) -- (d); 
        \draw (a) -- (d);
        \draw (c) -- (f);
        \draw (e) -- (d);
    \end{tikzpicture}
    \caption{The graph $R$ has approximate eigenvalues $\{2.247, 0.802, 0.555, -0.555, -0.802, -2.247\}$.}
    \label{fig: C4(0,0,1,1)}
\end{figure}
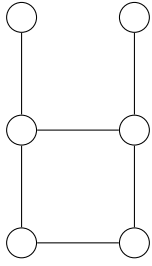
\begin{proof}
One can calculate the spectra of the graphs listed in Corollary~\ref{cor: forbidden subgraphs} to see that each graph either has its second largest eigenvalue at least 1 or has more than two positive eigenvalues, and so by Theorem~\ref{thm: interlacing} these graphs cannot be induced subgraphs to any graph cospectral to $P_2(1,n)$.
\end{proof}

 \begin{lemma} \label{2K2, R free}
     If $H$ is a connected bipartite graph that is $\{2K_2, R\}$-free, then $H$ is $\{P_4+K_1\}$-free. 
 \end{lemma}

 \begin{proof}
    Let $H$ be a connected bipartite graph that has no induced subgraph isomorphic to $2K_2$ or $R$. Assume towards contradiction that $H$ has an induced subgraph isomorphic to $P_4+K_1$, and label the vertices along the subgraph isomorphic to $P_4$ as $z_1z_2z_3z_4$. Then, there exists a vertex $w$ such that $w$ is not adjacent to any of the vertices $z_1$, $z_2$, $z_3$, or $z_4$. Since $H$ is connected, there exists a shortest path from $w$ to some vertex in the path $z_1z_2z_3z_4$. Let $t$ be the  vertex in this path that is adjacent to a $z_i$ for $i \in \{1,2,3,4\}$. By symmetry, it suffices to examine the cases where $t$ is adjacent to vertex $z_1$ or to vertex $z_2$.\\
    
    First, assume that $t$ is adjacent to $z_1$. Then $t$ cannot be adjacent to $z_2$ or $z_4$, since this would construct an odd cycle in $H$, which is a contradiction. If $t$ is adjacent to $z_1$ but not $z_3$, then dist($w,z_4$)$\ge 4$ and so $H$ has path $P_n$ with $n \ge 5$, within which there is an induced subgraph isomorphic to $2K_2$, a contradiction. If $t$ is adjacent to both $z_1$ and $z_3$, then $H$ has $R$ as an induced subgraph, as shown in Figure~\ref{fig: K22-1,1 contradiction}, which is again a contradiction.\\
      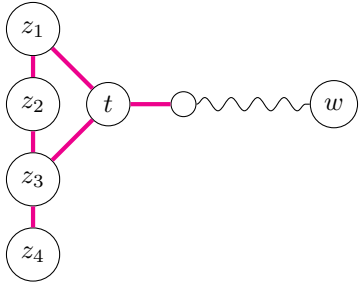
\begin{figure}[t]
    \centering
    \begin{tikzpicture}[main/.style={circle, draw}]
        \node[main](u) at (0,3){$z_1$};
        \node[main](x) at (0,2){$z_2$};
        \node[main](v) at (0,1){$z_3$};
        \node[main](y) at (0,0){$z_4$};
        \node[main](w) at (4,2){$w$};
        \node[main](term) at (1, 2){$t$}; 
        \node[main] (t2) at (2,2){};
        \draw[magenta, ultra thick ] (u) -- (x);
        \draw[magenta, ultra thick ] (x) -- (v);
        \draw[magenta, ultra thick ] (v) -- (y);
        \draw[magenta, ultra thick] (term) -- (u); 
        \draw[magenta, ultra thick ] (term) -- (v);
        \draw[magenta, ultra thick] (term) -- (t2); 
        \draw[decorate,decoration = snake] (t2) -- (w);
    \end{tikzpicture}
    \caption{The graph $R$ present as an induced subgraph of the graph $H$, a contradiction.}
    \label{fig: K22-1,1 contradiction}
\end{figure}

    Next, assume that $t$ is adjacent to $z_2$. Then $t$ cannot be adjacent to $z_1$ or $z_3$, since this would construct an odd cycle in $H$. By symmetry, the preceding paragraph also shows that we can assume $t$ is not adjacent to $z_4$. Then dist$(w,z_4) \ge 4$ and there is an induced path $P_n$ with $n \ge 5$, a contradiction. Thus we have shown that every nonisolated vertex $w$ in $H$ must be adjacent to at least one of $z_1$, $z_2$, $z_3$, or $z_4$. We  conclude that $H$ is $\{P_4 + K_1\} $-free. 
 \end{proof}

\begin{theorem} \label{thm: structure}
    Let $n\ge 3$, and let $G$ be a graph cospectral to $P_2(1,n)$. Then the nonisolated vertices in $V(G)$ can be partitioned into 4 nonempty vertex sets $A,B,C,D$ such that the induced subgraphs $G[A \cup B]$, $G[B \cup C]$, and $G[C \cup D]$ are all complete bipartite graphs, as shown in Figure~\ref{fig: structure}. 

\end{theorem}

\begin{proof}
    Let $G$ be a graph cospectral to $P_2(1,n)$ for $n\ge 3$. By Corollary~\ref{cor: forbidden subgraphs}, $G$ is $2K_2$-free and thus has one nontrivial component. Let $G'$ be the nontrivial component of $G$.
    \\
    
    We will show that the diameter of $G'$ is exactly 3. By Corollary~\ref{cor: forbidden subgraphs}, it follows that $G'$ is $P_5$-free, so $G$ has diameter less than or equal to $3$.
    Since $P_2(1,n)$ is bipartite, and $G$ is cospectral to $P_2(1,n)$, the spectrum of $G$ is symmetric about zero, and $G$ is bipartite. 
    Since $G'$ has two positive eigenvalues, it cannot be isomorphic to a complete bipartite graph, and so 
    has diameter 3 \cite{BABEL}. \\
    
    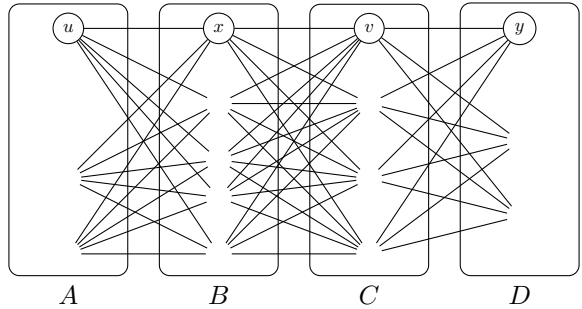
\begin{figure}[t]
    \centering
    \begin{tikzpicture}[main/.style= {draw, circle, scale = 0.7}, scale = 1]
        \node[main](u) at (-4,0){$u$}; 
        \node[main] (x) at (-2,0){$x$};
        \node[main] (v) at (0,0){$v$};
        \node[main] (y) at (2,0){$y$}; 
        \node[circle] (u1) at (-4.5, -1){}; 
        \node[circle] (u2) at (-4,-2 ){};
        \node[circle] (u3) at (-3.5,-1.5){};
        \node[circle] (u4) at (-4, -3){};
        \node [fit=(u1)(u2)(u3)(u4)(u),draw, rectangle, rounded corners, label = {below: $A$}] {};
        \node[circle] (x1) at (-2.5, -1){};
        \node[circle] (x2) at (-1.5, -1){};
        \node[circle] (x3) at (-2, -1){};
        \node[circle] (x4) at (-2, -1.75){};
        \node[circle] (x5) at (-2, -2.25){};
        \node[circle] (x6) at (-2,-3){};
        \node [fit=(x1)(x2)(x3)(x4)(x5)(x6)(x),draw, rectangle, rounded corners, label = {below: $B$}] {};
        \node[circle] (v1) at (-0.5,-1){};
        \node[circle] (v2) at (0.5, -1){};
        \node[circle] (v3) at (0,-1){};
        \node[circle] (v4) at (0,-2){}; 
        \node[circle] (v5) at (0,-3){};
        \node [fit=(v1)(v2)(v3)(v4)(v5)(v),draw, rectangle, rounded corners, label = {below: $C$}] {};
        \node[circle] (y1) at (1.5, -1){};
        \node[circle] (y2) at (2.5, -3){};
        \node[circle] (y3) at (2,-1.5){};
        \node[circle] (y4) at (2, -2.5){};
        \node [fit=(y1)(y2)(y3)(y4)(y),draw, rectangle, rounded corners, label = {below: $D$}] {};
        \draw (u) -- (x);
        \draw (x) -- (v);
        \draw (v) -- (y); 
        \draw (u) -- (x3);
        \draw (u) -- (x4);
        \draw (u) -- (x5);
        \draw (u) -- (x6);
        \draw (u2) -- (x3);
        \draw (u2) -- (x4);
        \draw (u2) -- (x5);
        \draw (u2) -- (x6);
        \draw (u4) -- (x3);
        \draw (u4) -- (x4);
        \draw (u4) -- (x5);
        \draw (u4) -- (x6);
        \draw (u2) -- (x);
        \draw (u4) -- (x);
        \draw (v3) -- (x3);
        \draw (v3) -- (x4);
        \draw (v3) -- (x5);
        \draw (v3) -- (x6);
        \draw (v4) -- (x3);
        \draw (v4) -- (x4);
        \draw (v4) -- (x5);
        \draw (v4) -- (x6);
        \draw (v5) -- (x3);
        \draw (v5) -- (x4);
        \draw (v5) -- (x5);
        \draw (v5) -- (x6);
        \draw (v) -- (x3);
        \draw (v) -- (x4);
        \draw (v) -- (x5);
        \draw (v) -- (x6);
        \draw (v3) -- (x);
        \draw (v4) -- (x);
        \draw (v5) -- (x);
        \draw (y3) -- (v3);
        \draw (y3) -- (v4); 
        \draw (y3) -- (v5); 
        \draw (y4) -- (v3);
        \draw (y4) -- (v4); 
        \draw (y4) -- (v5);
        \draw (y) -- (v3);
        \draw (y) -- (v4); 
        \draw (y) -- (v5);
        \draw (v) -- (y3); 
        \draw (v) -- (y4);
    \end{tikzpicture}
    \caption{General structure of a the nontrivial component $G'$ of a graph $G$ cospectral to $P_2(1,n)$. The induced subgraphs $G'[A \cup B]$, $G'[B \cup C]$, and $G'[C \cup D]$ are all complete bipartite graphs. }
    \label{fig: structure}
\end{figure}
    Hence there exists an induced subgraph isomorphic to $P_4$ in $G$. Let $P = uxvy$ be such a path. 
    We show that $G'$ has the form given in Figure~\ref{fig: structure}. \\

    By Lemma~\ref{2K2, R free}, every vertex $w$ in $G'$ has at least one neighbor on the path $P$. Note that since $G$ is $2K_2$-free, the neighborhood of $w$ cannot contain only $u$ or only $y$ from $\{u,x,v,y\}$. 
    Let $A$ be the set of vertices in $G'$ whose only neighbor in $P$ is $x$. Let $B$ be the set of vertices in $G'$ adjacent to both $u$ and $v$ on $P$. Let $C$ be the set of vertices in $G'$ such that each vertex in $C$ is adjacent to both $x$ and $y$ on $P$. And let $D$ be the set of vertices in $G'$ adjacent to only $v$ on $P$. 
    \\

    We show that each of the sets $A\cup B$, $B\cup C$, and $C\cup D$ induces a complete bipartite graph. We will omit the $C\cup D$ case, which is similar to $A \cup B$. Note that each vertex in $A \cup B$ is not adjacent to $y$. Then Lemma~\ref{2K2, R free} implies $G[A\cup B]$ is $P_4$-free. Similarly, $G[B\cup C]$ must be $P_4$-free, since $R$ is not an induced subgraph of $G[B\cup C \cup \{u\}]$.
    Since $G[A\cup B]$ and $G[B\cup C]$ are both $P_4$-free, $G[A\cup B]$ and $G[B\cup C]$ must both be complete bipartite graphs. We conclude that $G'$ must have the structure in Figure~\ref{fig: structure}. 
    \end{proof}
   The following lemmas give structural requirements for $G'$. 

    \begin{lemma} \label{lem: sachs term 4 thing}
     If $G$ is a bipartite graph, then the third nonzero term of its characteristic polynomial is the number of subgraphs in $G$ isomorphic to $2K_2$ whose vertices do not induce a graph isomorphic to $C_4$. 
 \end{lemma}
 \begin{proof}
     Let $G$ be a bipartite graph. We refer to a subgraph isomorphic to $2K_2$ as a 2-matching of $G$. 
     By Theorem~\ref{sachs}, the third nonzero coefficient in the characteristic polynomial is equivalent to $\begin{aligned} a_4 = \sum_{H \in \mathcal{G}_4} (-1)^{k(H)}(2)^{c(H)},\end{aligned}$ where $\mathcal{G}_4$ is the set of $4$-vertex subgraphs of $G$ in which each component is a cycle or is isomorphic to $K_2$, $k(H)$ denotes the number of components in the graph $H$, and $c(H)$ denotes the number of components which are cycles in $H$. Thus there are two types of graphs in $\mathcal{G}_4$: Graphs which are isomorphic to $C_4$ and graphs which are isomorphic to $2K_2$. Each copy of $C_4$ contributes $(-1)^1(2)^1 =-2$ to the sum, while every copy of $2K_2$ contributes $(-1)^2(2)^0 = 1$ to the sum $a_4$. \\

     There are two types of 2-matchings in $\mathcal{G}_4$: those which induce a subgraph isomorphic to $C_4$ and those which do not. Observe that for each distinct $C_4$ in $\mathcal{G}_4$, there are exactly two subgraphs of the form $2K_2$ in $\mathcal{G}_4$ on the same set of 4 vertices. That is, each set of 4 vertices in $G$ which induce a graph isomorphic to $C_4$, contributes $1+1-2=0$ to the sum $a_4$. Each set of 4 vertices which induce a graph isomorphic to $2K_2$ or $P_4$ contributes 1 to the sum $a_4$. Hence $a_4$ is equal to the total number of 2-matchings in $G$ which do not induce a subgraph of $G$ isomorphic to $C_4$. 
 \end{proof}

 Using the sets $A$, $B$, $C$, and $D$ as defined in the proof of Theorem~\ref{thm: structure},
 let $a= |A|$, $b= |B| $,  $c = |C|$; and let $d = |D| $. 

 \begin{lemma} \label{lem: char poly complete minus edge}
     Let $a,b,c,d \ge 1$. The graph $G'$ has characteristic polynomial $$\phi(G') = x^{a+b+c+d} - (ab + bc + cd)x^{a+b+c+d-2} +(abcd)x^{a+b+c+d-4}.$$
 \end{lemma}
 \begin{proof}
 First, observe that $G'$ has order $a+b+c+d$ and $ab+bc+cd$ edges, giving the first two nonzero terms of the characteristic polynomial by Theorem~\ref{sachs}. Since $G'$ is bipartite, there is no nonzero term with power 
 $a+b+c+d-1$ or $a+b+c+d-3$, which correspond to  forbidden subgraphs in all bipartite graphs. By Lemma~\ref{lem: sachs term 4 thing}, to calculate the coefficient of $x^{a+b+c+d-4}$, we count the subgraphs corresponding to 2-matchings whose vertices do not induce a graph isomorphic to $C_4$. 
 Let $v_a$ and $v_d$ be arbitrary vertices in $A$ and $D$ respectively. Observe that $v_a$ has degree $b$ and $v_d$ has degree $c$. Further observe that every 2-matching which does not include both a vertex in $A$ and a vertex in $D$ must induce a graph isomorphic to $C_4$, while no 2-matchings which include both $v_a$ and $v_d$ can induce a subgraph isomorphic to $C_4$ since $v_a$ and $v_d$ are not adjacent in $G'$. Thus the coefficient of $x^{a+b+c+d-4}$ will be the number of 2-matchings such that one edge must include a vertex from $A$ and the other must include a vertex from $D$. Since these vertices are nonadjacent and in different partite sets, we may choose any edge incident to $v_a$ and any edge incident to $v_d$. Considering all vertices in $A$ and $D$, there are $abcd$ such matchings. 
 \\

 Now $G$ is cospectral to $P_2(1,n)$ by assumption, so $G$ has exactly two nonzero positive eigenvalues, implying that there are no other nonzero terms in the characteristic polynomial $\phi(G')$. Hence we conclude $$\phi(G') = x^{a+b+c+d} - (ab +bc + cd)x^{a+b+c+d-2} + abcdx^{a+b+c+d-4}.$$ 
 \end{proof}

    The following theorem gives additional structural restrictions on $G'$.
   
\begin{theorem} \label{thm: restructure}
    Let $m \ge 2$. In the graph $G'$, one of the following holds:
    \begin{enumerate}[(i)]
        \item $c=1$ and $d=2$. That is, $G'$ can be constructed by identifying one vertex in the partite set of order 2 in a graph isomorphic to $K_{2,m}$ with the vertex in the partite set of order 1 in $K_{1,2}$.
        \item $b=2$ and $d=2$. That is, $G'$ is constructed from a complete bipartite graph $K_{4,m}$ with an additional vertex adjacent to two of the vertices in the partite set of order 4.
    \end{enumerate}
\end{theorem}
Note that the graphs $A_a$ and $B_a$ given in Section~\ref{sec: pairs} have the forms (i) and (ii) respectively. 
\begin{proof}
 If $a\ge 2$ and $d\ge 2$, then there is an induced $P_2(2,2)$ in the graph. By Corollary~\ref{cor: forbidden subgraphs}, such a graph cannot be cospectral to $P_2(1,n)$. So, without loss of generality, let $a=1$. \\ 
 
    We show that $d >1$ for all $G'$. Assume toward a contradiction that $d =1$. 
    Then by Lemma~\ref{lem: char poly complete minus edge}, $$\phi(G') = x^{b+c+2} - (b+bc+c)x^{b+c} +bcx^{b+c-2}.$$ 
    Equating coefficients with $\phi(P_2(1,n))$, we have the identities
    \begin{align*}
        n &= bc;\\
        b+bc+c &= n+2.
    \end{align*}
    That is, $b+c = 2$. Since $x \in B$ and $v \in C$, both $b$ and $c$ are at least 1. Therefore $b=1$ and $c=1$, so the product $bc=1$. But, since $bc = n$ and $n$ is assumed to be at least 3, this is a contradiction. \\
    
   

    So, assume $d \ge 2$. By Theorem~\ref{double_star_spec} and Lemma~\ref{lem: char poly complete minus edge}, 
    \begin{align*}
        \phi(P_2(1,n)) &= x^{n+3} - (n+2)x^{n+1} +nx^{n-1}\\
        \phi(G') &= x^{b+c+d+1} - (b+bc+cd)x^{b+c+d-1} + bcdx^{b+c+d-3}
    \end{align*}
    
    Then by equating coefficients in the characteristic polynomials, we have the following two identities 
    \begin{align*}
        n&= bcd;\\
        n+2 & = b+bc+cd.
    \end{align*}
    It follows that 
    \begin{align*}
        bcd+2&= b+bc+cd; \\
        b+bc -2 &\ge 2c(b-1);\\
        c(2-b) &\ge 2-b.
    \end{align*}
    \begin{figure}[t]
    \centering
    \begin{tikzpicture}[main/.style= {draw, circle, scale =0.7}, scale = 1]
        \node[main](u) at (-4,0){$u$}; 
        \node[main] (x) at (-2,0){$x$};
        \node[main] (v) at (0,0){$v$};
        \node[main] (y) at (2,0){$y$}; 
        \node[circle] (u1) at (-4.5, -1){}; 
        \node[circle] (u2) at (-4,-2 ){};
        \node[circle] (u3) at (-3.5,-1.5){};
        \node[circle] (u4) at (-4, -3){};
        \node [fit=(u1)(u2)(u3)(u4)(u),draw, rectangle, rounded corners, label = {below: $A$}] {};
        \node[circle] (x1) at (-2.5, -1){};
        \node[circle] (x2) at (-1.5, -1){};
        \node[circle] (x3) at (-2, -1){};
        \node[circle] (x4) at (-2, -1.75){};
        \node[circle] (x5) at (-2, -2.25){};
        \node[circle] (x6) at (-2,-3){};
        \node [fit=(x1)(x2)(x3)(x4)(x5)(x6)(x),draw, rectangle, rounded corners, label = {below: $B$}] {};
        \node[circle] (v1) at (-0.5,-1){};
        \node[circle] (v2) at (0.5, -1){};
        \node[circle] (v3) at (0,-1){};
        \node[circle] (v4) at (0,-2){}; 
        \node[circle] (v5) at (0,-3){};
        \node [fit=(v1)(v2)(v3)(v4)(v5)(v),draw, rectangle, rounded corners, label = {below: $C$}] {};
        \node[circle] (y1) at (1.5, -1){};
        \node[circle] (y2) at (2.5, -3){};
        \node[main, scale = 1.5] (y3) at (2,-1.5){};
        \node[circle] (y4) at (2, -2.5){};
        \node [fit=(y1)(y2)(y3)(y4)(y),draw, rectangle, rounded corners, label = {below: $D$}] {};
        \draw (u) -- (x);
        \draw (x) -- (v);
        \draw (v) -- (y); 
        \draw (u) -- (x3);
        \draw (u) -- (x4);
        \draw (u) -- (x5);
        \draw (u) -- (x6);
        \draw (v) -- (x3);
        \draw (v) -- (x4);
        \draw (v) -- (x5);
        \draw (v) -- (x6);
        \draw (v) -- (y3); 
    \end{tikzpicture}
    \caption{The graph $G'$ such that $c=1$, $d=2$, and $b$ is greater than 2. This corresponds to structure (i) given in Theorem~\ref{thm: structure}.}
    \label{fig: structure c=1}
\end{figure}
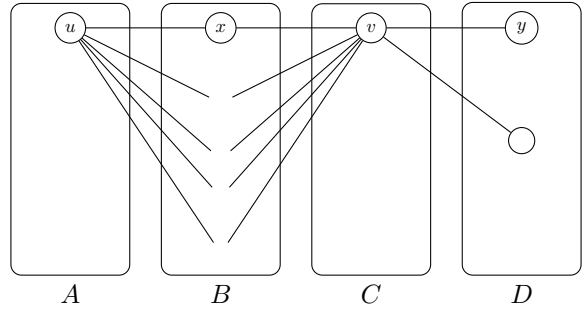

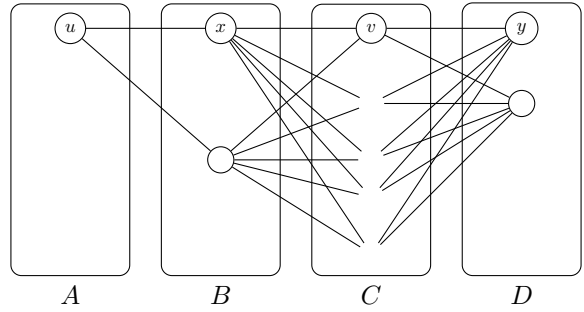
\begin{figure}[t]
    \centering
    \begin{tikzpicture}[main/.style= {draw, circle, scale = 0.7}, scale = 1]
        \node[main](u) at (-4,0){$u$}; 
        \node[main] (x) at (-2,0){$x$};
        \node[main] (v) at (0,0){$v$};
        \node[main] (y) at (2,0){$y$}; 
        \node[circle] (u1) at (-4.5, -1){}; 
        \node[circle] (u2) at (-4,-2 ){};
        \node[circle] (u3) at (-3.5,-1.5){};
        \node[circle] (u4) at (-4, -3){};
        \node [fit=(u1)(u2)(u3)(u4)(u),draw, rectangle, rounded corners, label = {below: $A$}] {};
        \node[circle] (x1) at (-2.5, -1){};
        \node[circle] (x2) at (-1.5, -1){};
        \node[circle] (x3) at (-2, -1){};
        \node[main, scale = 1.5] (x4) at (-2, -1.75){};
        \node[circle] (x5) at (-2, -2.25){};
        \node[circle] (x6) at (-2,-3){};
        \node [fit=(x1)(x2)(x3)(x4)(x5)(x6)(x),draw, rectangle, rounded corners, label = {below: $B$}] {};
        \node[circle] (v1) at (-0.5,-1){};
        \node[circle] (v2) at (0.5, -1){};
        \node[circle] (v3) at (0,-1){};
        \node[circle] (v4) at (0,-1.75){}; 
        \node[circle] (v5) at (0,-2.25){};
        \node[circle] (v6) at (0, -3){};
        \node [fit=(v1)(v2)(v3)(v4)(v5)(v6)(v),draw, rectangle, rounded corners, label = {below: $C$}] {};
        \node[circle] (y1) at (1.5, -1){};
        \node[circle] (y2) at (2.5, -3){};
        \node[main, scale = 1.5] (y3) at (2,-1){};
        \node[circle] (y4) at (2, -1.5){};
        \node[circle] (y5) at (2, -2){};
        \node[circle] (y6) at (2, -2.5){};
        \node[circle] (y7) at (2,-3){};
        \node [fit=(y1)(y2)(y3)(y4)(y),draw, rectangle, rounded corners, label = {below: $D$}] {};
        \draw (u) -- (x);
        \draw (x) -- (v);
        \draw (v) -- (y); 
        \draw (v) -- (y3);
        \draw (u) -- (x4);
        \draw (v) -- (x4);
        \draw (v3) -- (x4); 
        \draw (v4) -- (x4); 
        \draw (v5) -- (x4); 
        \draw (v6) -- (x4); 
         \draw (v3) -- (x); 
        \draw (v4) -- (x); 
        \draw (v5) -- (x); 
        \draw (v6) -- (x); 
        \draw (v3) -- (y3); 
        \draw (v4) -- (y3); 
        \draw (v5) -- (y3); 
        \draw (v6) -- (y3); 
         \draw (v3) -- (y); 
        \draw (v4) -- (y); 
        \draw (v5) -- (y); 
        \draw (v6) -- (y); 
    \end{tikzpicture}
    \caption{The graph $G'$ with $b=2$, $d=2$, and $c$ greater than 1. This graph has structure (ii) defined in Theorem~\ref{thm: structure}. }
    \label{fig: structure b=2}
\end{figure}

    When $b > 2$, this gives the inequality $c \le 1$.  Since $c \ne 0$, assume $c= 1$ as shown in Figure~\ref{fig: structure c=1}. Further, note that since $c=1$, \begin{align*}
        bcd+2 &= b + bc+cd; \\
        d&=2.
    \end{align*}
    Hence $G'$ has form (i) in the theorem statement.\\

    Next we examine the cases where $b=1$ and $b=2$. First, assume $b=1$. Then we have the identity
    \begin{align*}
        cd+2 &= 1 +c+cd;\\
        1&=c.
    \end{align*}
    Such a graph is isomorphic to $P_2(1,d)$. To be cospectral to $P_2(1,n)$, it must be the case that $d=n$, and this graph is isomorphic to our original graph. \\

    Lastly, assume $b=2$. Then 
    \begin{align*}
        2cd +2 &= 2 + 2c +cd; \\
        d&=2.
    \end{align*}
    
    This gives the graph in Figure~\ref{fig: structure b=2}, which has the form (ii) as given in the theorem statement. 
\end{proof}

\begin{theorem}
    If there exists a graph $G$ cospectral to $P_2(1,n)$ but not isomorphic to $P_2(1,n)$, then $n$ is even. 
\end{theorem}

\begin{proof}
    Assume that there exists a graph $G$ cospectral to $P_2(1,n)$ but not isomorphic to $P_2(1,n)$. By Theorems~\ref{thm: structure} and \ref{thm: restructure}, we know that $G$ has  either form (i) or form (ii).\\

    We know that $P_2(1,n)$ has characteristic polynomial $\phi(P_2(1,n)) = x^{n+3} - (n+2)x^{n+1} +nx^{n-1}$.
    First, assume that $G$ has form (i). From the proof of Theorem~\ref{thm: P_2(1,even)}, such a graph has characteristic polynomial $x^{b+4}-(2b+2)x^{b+2}+2bx^b$. Since the graphs have equal characteristic polynomials, it must be that $n=2b$, and hence $n$ is even. \\

    Assume now that $G$ has form (ii). $G$ has characteristic polynomial $x^{5+c} -(4c+2)x^{3+c}+4cx^{1+c}$, so $n = 4c$, and again $n$ is even. 
\end{proof}

This brings us to our main result. 
\begin{theorem}\label{thm: big boy}
    The double star $P_2(1,n)$ is DS if and only if $n$ is odd or $n=2.$
\end{theorem}

\section{Conclusion}
In this paper, we have characterized all graphs that share a spectrum with $P_2(1,n)$, concluding that, though almost all trees are not DS, $P_2(1,n)$ is DS for odd values of $n$. A natural extension to this question would be to examine slightly larger double stars to determine conditions on $a$ and $b$ such that $P_2(a,b)$ is DS. 
Increasing $a$ from one while allowing $b$ to vary quickly produces interestingly large sets of nonisomorphic cospectral graphs. For example, $P_2(4,4)$, $P_2(4,6)$, and $P_2(1,4a)$ for $a \ge 2$ each share a spectrum with two distinct graphs. 
$P_2(2,9)$, $P_2(3,4)$, $P_2(5,6)$, and $P_2(4,5)$ each have three graphs to which they are cospectral. The double star $P_2(4,7)$ has at least four cospectral mates, and $P_2(3,8)$ has at least five graphs to which it is cospectral. 
We conjecture that $P_2(4,k)$ for $k>2$ has multiple nonisomorphic cospectral mates. \\

As stated in Theorem~\ref{thm:pairs}, we have given a construction for a graph cospectral to $P_2(2,n)$ when $n$ is odd and a construction for a graph cospectral to $P_2(3,m)$ when $m$ is even. Small examples suggest that neither graph $P_2(2,n)$ for even $n$ or $P_2(3,m)$ for $m$ odd will be DS. We conjecture that there will be common substructures that must appear in graphs cospectral the these double stars, similar to the substructures detailed in Section~\ref{sec: 1,a} for $P_2(1,n)$. Most generally, we conjecture that graphs cospectral to double stars will (eventually) contain $K_{2,k}$ as a subgraph for some conditions on $k$.  \\

It is tempting to conjecture that $P_2(1,n)$ for odd values of $n$ are the only DS double stars, but we recall that $P_2(5,5)$ is DS. 
We propose that determining if $P_2(5,5)$ is the smallest graph in a generalizable pattern of DS double stars is an interesting question for further research. 

\printbibliography

\end{document}